\documentclass[12pt, oneside]{amsart}   	
\usepackage{geometry}
\usepackage{amsmath, amssymb}
\usepackage{epsfig}
\usepackage{epic,eepic}              		
\usepackage{graphicx}				
\newtheorem{thm}{Theorem}[section]

\newtheorem{lem}[thm]{Lemma}

\newtheorem{rem}[thm]{Remark}

\theoremstyle{definition}
\newtheorem{defn}[thm]{Definition}

\numberwithin{equation}{section}

\title{A proof of the Multijoints Conjecture and Carbery's generalization}
\author{Ruixiang Zhang}
\begin{document}

\maketitle
\begin{abstract}
We present a new proof of the Joints Theorem without taking derivatives. Then we generalize the proof to prove the Multijoints Conjecture and Carbery's generalization. All results are in any dimension over an arbitrary field.
\end{abstract}

\section{Introduction}

A recent major application of the polynomial method in discrete geometry is the proof of the following Joints Theorem. It was initially proved by Guth and Katz \cite{guth2010algebraic} in $\mathbb{R}^3$ and then generalized to $\mathbb{R}^d$ by Quilodr{\'a}n \cite{quilodran2010joints} and Kaplan-Sharir-Shustin \cite{kaplan2010lines}. For the arbitrary field case $\mathbb{F}^d$ it is also known by the work of Carbery-Iliopoulou\cite{carbery2014counting}.

\begin{thm}[Joints Theorem \cite{guth2010algebraic}\cite{quilodran2010joints}\cite{kaplan2010lines}\cite{carbery2014counting}]\label{jointsthm}
For $N$ lines in $\mathbb{F}^d$ ($d \geq 2$), a \emph{joint} is an intersection of $d$ lines such that the directions of the $d$ lines are linearly independent. Then there are $\lesssim N^{\frac{d}{d-1}}$ joints.
\end{thm}

The joint theorem states an inequality of multilinear flavor and can be viewed as a simplified discrete counterpart of the Multilinear Kakeya theorem proved by Bennett-Carbery-Tao \cite{bennett2006multilinear} (non-endpoint case) and Guth \cite{guth2010endpoint} (full conjecture). Being a central topic in harmonic analysis, multilinear Kakeya states similar bounds to Theorem \ref{jointsthm} on intersections of $1$-tubes rather than lines. However, it says ``more'' than Theorem \ref{jointsthm} in the sense that it also takes the multiplicities of tubes into account (hence of a ``Kakeya maximal inequality'' flavor). Based on the multilinear Kakeya and people's success of proving the Joints Theorem, Carbery brought up the concept of ``multijoints'', which is a straight generalization of the concept of joints (see Theorem \ref{multijoint}). He then made a general conjecture on upper bounding the number of multijoints with multiplicity in any finite dimensional vector space over any field as a discrete analogue of the multilinear Kakeya. In this paper we prove this conjecture and prove its important special case known as the Multijoints Conjecture (also brought up by Carbery) along the way. Namely we will prove:

\begin{thm}\label{multijoint}[Multijoints Theorem]
Let $\mathbb{F}$ be a field. In $\mathbb{F}^d$, assuming we are given $d$ families of lines $L_1, \ldots, L_d$ such that $N_i = |L_i| > 0$. A multijoint is an intersection of $d$ lines $l_i \in L_i, 1\leq i \leq d$ such that the directions of the $d$ lines are linearly independent. Then there are $\lesssim (\prod_{i=1}^d N_i)^{\frac{1}{d-1}}$ multijoints.
\end{thm}

\begin{thm}\label{carberytheorem}[Carbery's conjecture holds]
In $\mathbb{F}^d$, assuming we are given $d$ families of lines $L_1, \ldots, L_d$ such that $N_i = |L_i| > 0$. For every point $P \in \mathbb{F}^d$, define the multijoint multiplicity $\mu(P) = |\{(l_1, \ldots, l_d): l_i \in L_i, \{l_i\} \text{ form a multijoint at } P\}|$. Then
\begin{equation}\label{carberytheoremeq}
\sum_{P \in \mathbb{F}^d} \mu(P)^{\frac{1}{d-1}} \lesssim (\prod_{i=1}^d N_i)^{\frac{1}{d-1}}.
\end{equation}
\end{thm}

As stated above, Theorem \ref{carberytheorem} (as well as its generalization Theorem \ref{multiplicityjointsthm} that we shall see a bit later) can be viewed as a discrete analogue of the Multilinear Kakeya Theorem.

Important partial results on Theorem \ref{multijoint} and Theorem \ref{carberytheorem} has been made in \cite{iliopoulou2013counting}\cite{iliopoulou2013discrete}\cite{carbery2014counting}\cite{carbery2014colouring}\cite{iliopoulou2015counting}\cite{iliopoulou2015incidence}\cite{hablicsek2014joints}\cite{yang2016generalizations}. For example in \cite{iliopoulou2015incidence}, Theorem \ref{multijoint} was shown to be true for $\mathbb{R}^d$ or $\mathbb{F}^3$ for a general field $\mathbb{F}$. The partial results in the paper we list above have various restrictions on the field, the dimension, the type of multijoints (assuming the multijoints are \emph{generic}, which means whenever $d$ lines from the $d$ families meet they form a multijoint, is usually helpful), the exponent (In \cite{yang2016generalizations} Theorems \ref{multijoint} and \ref{carberytheorem} in $\mathbb{R}^d$ up to endpoint on the exponent was shown, among other generalizations). We do not state a comprehensive set of known results here and refer the interested readers to the above references.

Our approach, as with all current proofs of the Joints Theorem, uses the polynomial method whose initial application in incidence geometry was in Dvir's proof of the finite field Kakeya Conjecture \cite{dvir2009size}. A somewhat unique feature of our proof can be briefly described in the following way: In the ``classical'' polynomial method in discrete geometry, we usually force the polynomial to vanish on a lot of lines as a result of the polynomial being heavily incident to the line. We then usually have to take special care to show that a nonzero polynomial could not vanish on too many lines in the problem and hence deduce a contradiction. On the other hand, in our proof we mainly look at the Taylor series of the polynomial around each multijoint rather than caring if it actually vanishes at the point. As a result, the lines where our polynomial vanish do not need to be specially taken care of, and could be handled in exactly the same way as we handle the other lines where the polynomial does not vanish. For example along the way of our proof of Theorems \ref{multijoint} and \ref{carberytheorem}, we are able to define a ``generalized intersection multiplicity'' of a line $l$ and a polynomial hypersurface $Z(Q)$ at a point $P \in l \bigcap Z(Q)$ even if $l \subseteq Z(Q)$. And for this generalized intersection multiplicity we have an upper bound estimate as good as what B\'{e}zout's Theorem gives (see Section 3 for details).

The ``Taylor series'' viewpoint here was already around for some time (\cite{carberycomm}), while it seems to me that the ``generalized intersection multiplicity'' we defined above is a relatively new concept. I believe the results in the current paper shows that both are interesting in their own rights and are nice additions to our current toolbox of the polynomial method. At least we see its use in the multilinear setting of the incidence problems through the problems here.

For future directions, it is conceivable that one may generalize Theorems \ref{multijoint} and \ref{carberytheorem} to the setting where we have joints formed by general algebraic curves (which was considered in \cite{iliopoulou2013counting}\cite{iliopoulou2013discrete}\cite{iliopoulou2015counting}\cite{iliopoulou2015incidence}\cite{yang2016generalizations}). Another more interesting direction is to try to prove the similar upper bounds of joints formed by higher dimensional objects (which should be the endpoint and $\mathbb{F}^d$ case of the results in \cite{yang2016generalizations}). Maybe this can be done by generalizing the tools we have here and probably those we developed in \cite{zhang2015endpoint} for handling the interaction between multiple hypersurfaces and a high dimensional object. We do not address either topic here and leave them to interested readers.

In this paper, we fix the dimension $d>1$. All the implied constants will depend on $d$ and we hence suppress the dependence on $d$ in our notations. For example, ``$\lesssim$'' should really be understood as ``$\lesssim_d$''.

\section*{Acknowledgements}

I was supported by Princeton University and the Institute for Pure and Applied Mathematics (IPAM) during the research. Part of this research was performed while I was visiting IPAM, which is supported by the National Science Foundation. I thank IPAM for their warm hospitality and amazing program. I would like to thank Anthony Carbery, Jordan Ellenberg, Larry Guth, Marina Iliopoulou, Yuchen Liu, Ben Yang and Ziquan Zhuang for helpful discussions.

\section{A new proof of the Joints Theorem}

In this section we present a new proof of the Joints Theorem \ref{jointsthm} that could be generalized to prove the Multijoints Theorem and Carbery's conjecture as we will see a bit later. In this proof we do not take any kind of derivatives hence it works the same way on linear spaces over every field.

Let $\mathbb{F}$ be an arbitrary field. As a convention, by a \emph{nonzero polynomial} over $\mathbb{F}^d$ we mean a polynomial whose coefficients are not all zero, even though it might vanish at every point in the whole $\mathbb{F}^d$. We say $Q \neq 0$ if $Q$ is a nonzero polynomial and shall continue using this notion throughout this paper. First we state the following variant of the well-known parameter counting lemma in linear algebra, which is immediate from the theory of linear equations.

\begin{lem}[Parameter Counting]
Write any polynomial $Q$ of $d$ variables as $Q = \sum_{\beta = (\beta_1, \ldots, \beta_d)} c_{\beta} x^{\beta}$ where almost all $c_{\beta} = 0$. For any finite number $A$ and $A$ homogeneous linear forms $H_1, \ldots, H_A$, each supported on finitely many $c_{\beta}$ as variables, there exists a nonzero polynomial $Q_0$ of degree $\lesssim A^{\frac{1}{d}}$ such that $H_1, \ldots, H_A$ all vanish at the coefficients of $Q_0$.
\end{lem}

\begin{proof}[Proof of Theorem \ref{jointsthm}]
Assume that $J$ is the set of joints and $L$ the set of given lines. By parameter counting, we can find a nonzero polynomial $Q$ of degree $\lesssim |J|^{\frac{1}{d}}$ that vanishes on each joint. Now for each line $l \in L$, we define the concept of ``ordinary'' and ``special'' joints (with respect to $Q$) on it: For any joint $P \in l$, there exists an affine linear transform $T$ that sends $P$ to the origin, $l$ to the $x_d$-axis, and $Q$ to a new function $(T^{-1})^* Q$ which is a polynomial of degree same as $Q$. We call $P$ an \emph{ordinary} joint on $l$, if the lowest homogeneous term of $(T^{-1})^* Q$ is independent of $x_d$. Otherwise $P$ is called a \emph{special} joint on $l$.

The above definition is intrinsic, i.e. independent of the transform $T$. In fact, if there is another such linear transform $T'$. Then $T' = T_1 \circ T$ where $T_1$ is a scaling when restricted to the $x_d$-axis. Such $T_1$ has the form $(T_1^{-1})^* (x_1, \ldots, x_d) = (h_1 (x_1, \ldots, x_{d-1}), \ldots, h_{d-1} (x_1, \ldots, x_{d-1})), \lambda x_d + h_d (x_1, \ldots, x_{d-1})$ where $\lambda \neq 0$, $h_1, \ldots, h_d$ are linear functions and the transform $(h_1, \ldots, h_{d-1})$ is invertible. Now $({T'}^{-1})^* Q = (T_1^{-1})^* (T^{-1})^* Q$. By the explicit form of $(T_1^{-1})^*$ we obtained above, we see that the definition of ordinarity/speciality does not depend on whether we choose the linear transform to be $T$ or $T'$.

Next we notice that a joint cannot be ordinary on all lines passing through it, since we can take a linear transform to transform the $d$ transversal lines passing through it to the coordinate axes and $Q \neq 0$. Then the lowest order homogeneous term of the new polynomial has to depend on some variable. Hence $P$ is special with respect to the corresponding line.

Let us prove that on any line $l$ there are $\leq \deg Q$ special joints. We take a linear transform $T$ that sends $l$ to $x_d$ axis. Now the polynomial $Q$ under the new coordinate system have the form $Q = \sum_{\alpha} f_{\alpha} (x_d) x^{(\alpha, 0)}$. Here by definition, $\alpha = (\alpha_1, \alpha_2, \ldots, \alpha_{d-1})$ is a $(d-1)$-dimensional  multi-index and $x^{(\alpha, 0)} = x_1^{\alpha_1} x_2^{\alpha_2} \cdots x_{d-1}^{\alpha_{d-1}}$. We denote $|\alpha| = \sum_{i} \alpha_i$. Now we find a minimal $|\alpha_0|$ such that $ f_{\alpha_0} (x_d) \neq 0$ and claim that all special joints mush have their  $x_d$ coordinate being a root of $f_{\alpha_0}$. Indeed, if $P \in l$ is a joint such that its $x_d$ coordinate $x_d (P)$ is not a root of $f_{\alpha_0}$, then the lowest homogeneous term of $(T^{-1})^* Q$ will include a nonzero monomial $f_{\alpha_0} (x_d (P)) x^{(\alpha_0, 0)}$, since $|\alpha_0|$ is the smallest among all $|\alpha|$. Now all monomials in this lowest homogeneous term have to be independent of $x_d$, since otherwise the total power of $x_1, \ldots, x_{d-1}$ in some monomial would be smaller than $|\alpha_0|$ (a contradiction). Hence $P$ is ordinary on $l$.

Since each joint must be special with respect to at least one line, and on each line there are $\leq \deg Q$ special joints, we have $|J| \leq N \deg Q \lesssim N |J|^{\frac{1}{d}}$. Thus $|J| \lesssim N^{\frac{d}{d-1}}$.
\end{proof}

\section{Some definitions motivated by the new proof of the Joints Theorem}

The key idea in the proof of Theorem \ref{jointsthm} is not to be scared of the vanishing of a polynomial on a line, and use the ``ordinarily'' of most points on a line to proceed. We will generalize this idea to prove the Multijoints Conjecture (Theorem \ref{multijoint}) and Carbery's generalization (Theorem \ref{carberytheorem}) for all fields in all dimensions.

We prepare some tools for the proofs. In order to deal with multijoints problems, we would like to generalize the concept of ``ordinarity/speciality'' to take multiplicity into account. The definition we end up using is slightly different from a direct generalization of ``ordinarity/speciality'' in the last section.

Let $\mathbb{F}$ be an arbitrary field. In $\mathbb{F}^d$ assuming we have a point $P$ on a line $l$ and a polynomial $Q$. We define the \textbf{$(P, l)$-multiplicity of $Q$} in the following way:

We choose a linear transform $T$ that sends $l$ to the $d$-th coordinate axis. Then we can write
\begin{equation}\label{defnofmPl}
(T^{-1})^* Q = \sum_{\alpha = (\alpha_1, \ldots, \alpha_{d-1})} x_1^{\alpha_1} \cdots x_{d-1}^{\alpha_{d-1}} f_{\alpha} (x_d) = \sum_{\alpha = (\alpha_1, \ldots, \alpha_{d-1})} x^{(\alpha, 0)} f_{\alpha} (x_d).
\end{equation}

In the above expansion, look at all $(d-1)$-dimensional indices $\alpha$ with $f_{\alpha} \neq 0$ and $|\alpha| = \alpha_1 + \cdots + \alpha_{d-1}$ being the smallest possible. Call all such tuples $\alpha$ to be \emph{lowest} (with respect to $(P, l, T)$). Assuming $T (P) = (0, \ldots, 0, p_T)$. Look at all lowest tuples $\alpha$ and corresponding $f_{\alpha}$. Assuming $p_T$ is a root of $f_{\alpha}$ with multiplicity $m_{\alpha}$ (which is allowed to be zero). Then we define the \emph{$(P, l)$-multiplicity of $Q$} to be the minimal $m_{\alpha}$ among all lowest $\alpha$.

We check that this is a well-defined quantity largely similar to what we did in the last section. In fact, if we have another linear transform $T'$ sending $l$ to the $d$-th coordinate axis and $P$ to $(0, \ldots, 0, p_{T'})$, then $T' = T_1 \circ T$ where $T_1$ is a non-degenerate linear transform when restricted to the $x_d$-axis. Such $T_1$ has the form $(T_1^{-1})^* (x_1, \ldots, x_d) = (h_1 (x_1, \ldots, x_{d-1}), \ldots, h_{d-1} (x_1, \ldots, x_{d-1})), h(x_d) + h_d (x_1, \ldots, x_{d-1})$ where $h, h_1, \ldots, h_d$ are linear functions and the transform $(h_1, \ldots, h_{d-1})$ as well as $h$ is invertible. Moreover we have $h(p_T) = p_{T'}$.

As before we have $({T'}^{-1})^* Q = (T_1^{-1})^* (T^{-1})^* Q$. Now we have $(T_1^{-1})^* (x^{(\alpha, 0)} f_{\alpha} (x_d)) = (h_1 (x_1, \ldots, x_{d-1}), \ldots, h_{d-1} (x_1, \ldots, x_{d-1}))^{\alpha} f_{\alpha} (h(x_d) + h_d (x_1, \ldots, x_{d-1}))$. Since $(h_1, \ldots, h_{d-1})$ is non-degenerate, all $(h_1 (x_1, \ldots, x_{d-1}), \ldots, h_{d-1} (x_1, \ldots, x_{d-1}))^{\alpha}$ are linearly independent. Hence the new lowest $|\alpha'|$ (with respect to $(P, l, T')$) is the same as the lowest $|\alpha|$ (with respect to $(P, l, T)$). When explicitly computing the $(P, l)$-multiplicity of $Q$ under $T'$, we may ignore the contribution of $h_d (x_1, \ldots, x_{d-1})$ altogether since they only produces $|\alpha'|$ larger than the lowest one. It is then straightforward to see that (by the linear independence of all $(h_1 (x_1, \ldots, x_{d-1}), \ldots, h_{d-1} (x_1, \ldots, x_{d-1}))^{\alpha}$ for lowest $\alpha$ with respect to $T$) the $(P, l)$-multiplicity of $Q$ does not depend on whether we choose the linear transform to be $T$ or $T'$.

For our convenience denote $m_Q (P, l) \geq 0$ to be the $(P, l)$-multiplicity of $Q$.

\begin{rem}
When $Q$ is not identically zero on $l$, $m_Q (P, l)$ is equal to the intersection multiplicity between $Q$ and $l$. However, when $Q$ vanishes on $l$ we still have this $m_Q (P, l) < \infty$. Note that under this situation it is still true that for almost all $P \in l$, $m_Q (P, l)=0$.
\end{rem}

In terms of $(P, l)$-multiplicity, we have an inequality in place of B\'{e}zout's Theorem of intersection multiplicity.

\begin{lem}\label{bezoutineq}
\begin{equation}\label{bezoutineqeqn}
\sum_{P \in l} m_Q (P, l) \leq \deg Q.
\end{equation}
\end{lem}

\begin{proof}
Without loss of generality, we may assume $l$ is the $x_d$-axis. Write
\begin{equation}
Q = \sum_{\alpha = (\alpha_1, \ldots, \alpha_{d-1})} x^{(\alpha, 0)} f_{\alpha} (x_d).
\end{equation}

We choose a lowest $\alpha = \alpha_0$. Then
\begin{equation}
\sum_{P \in l} m_Q (P, l) \leq \sum_{P = (0, \ldots, 0, p) \in l} \text{ the multiplicity of } p \text{ as a root of } f_{\alpha_0} \leq \deg f_{\alpha_0} \leq \deg Q.
\end{equation}
\end{proof}

In order to obtain a lower bound of $m_Q (P, l)$ we will often invoke the following very strong lemma.

\begin{lem}\label{lowerbdofm}
Assuming that a linear transform $T$ sends $P$ to $(0, \ldots, 0, p_T)$ and $l$ to the $x_d$-axis. We look at the Taylor expansion of $(T^{-1})^* Q$ at $(0, \ldots, 0, p_T)$:
\begin{equation}\label{decompositionintomonomials}
(T^{-1})^* Q = \sum_{\beta = (\beta_1, \ldots, \beta_d)} c_{\beta} x_1^{\beta_1} \cdots x_{d-1}^{\beta_{d-1}} \cdot (x_d-p_T)^{\beta_d}.
\end{equation}

Assuming some $\beta_0 = (\beta_{0, 1}, \ldots, \beta_{0, d})$ satisfying that $c_{\beta_0} \neq 0$ and that among all $\beta$ with $c_{\beta} \neq 0$, $|\beta_0|$ is the smallest possible. Then
\begin{equation}\label{monomiallowerboundofmultiplicity}
m_Q (P, l) \geq \beta_{0, d}.
\end{equation}
\end{lem}

\begin{proof}
Look at all possible $\beta = \beta' = (\beta_1 ', \ldots, \beta_d ')$ in (\ref{decompositionintomonomials}) such that $c_{\beta'} \neq 0$ and $\beta_1 ' + \cdots + \beta_{d-1} '$ is smallest possible. We must have $\sum_{j=1}^{d-1}\beta_j ' \leq \sum_{j=1}^{d-1} \beta_{0, j}$. But $|\beta'| \geq |\beta_0|$ by assumption. Thus $\beta_d ' \geq \beta_{0, d}$.

Now merge the terms in (\ref{decompositionintomonomials}) to have the form (\ref{defnofmPl}). We see that all the lowest $\alpha$ will have its $f_{\alpha}$, when expanded as Taylor series at $P$, having every term divisible by $(x_d - p_T)^{\beta_{0, d}}$. Hence $f_{\alpha}$ is divisible by $(x_d - p_T)^{\beta_{0, d}}$ and by definition (\ref{monomiallowerboundofmultiplicity}) holds.
\end{proof}

\begin{rem}
Lemma \ref{lowerbdofm} is strong in the following sense: Despite the fact that $m_Q (P, l)$ is independent of the choice of $T$, by this lemma we can bound it from below when only given some (very incomplete) information of any fixed $T$.
\end{rem}

\section{A proof of the Multijoints Conjecture}

Theorem \ref{multijoint} is a special case of Theorem \ref{carberytheorem}. However we have a very simple proof for Theorem \ref{multijoint} by the tools we have developed. Hence we present this proof in a separate section before proving the harder Theorem \ref{carberytheorem}.

\begin{proof}[Proof of Theorem \ref{multijoint}]
Assume that $J$ is the set of multijoints. It has finitely many points. For each $P \in J$, choose $l_{i, P} \in L_i (1 \leq i \leq d)$ such that $l_{i, P}$ all pass through $P$ and have their directions span $\mathbb{F}^d$. Choose an affine linear transform $T_P$ to transform $P$ to the origin and $l_{i, P}$ to the $i$-th coordinate axis ($1 \leq i \leq d$).

By parameter counting we deduce that there exists a nonzero polynomial $Q$ of degree $\lesssim |J|^{\frac{1}{d}}\cdot (\prod_{i=1}^d |L_i|)^{\frac{1}{d}}$ such that: For each point $P \in J$, $(T_P^{-1})^*Q$, when expanded as a sum of monomials, has no term $x_1^{\beta_1} \cdots x_d^{\beta_d}$ with $\beta_i \leq |L_i|$ simultaneously holding.

For any $P\in J$, assuming $c_{\beta} x^{\beta}$ is a nonzero monomial in $(T_P^{-1})^* (Q)$ such that $|\beta| = \beta_1 + \cdots \beta_d$ is the smallest (here $\beta = (\beta_1, \ldots, \beta_d)$). By the assumption on $Q$ above, at least one of $\beta_i > |L_i|$ holds. If there is such a $\beta$ s.t. $\beta_i > |L_i|$ holds we say that $P$ is of type $i$. By Lemma \ref{lowerbdofm} we have that for any $P \in J$ of type $i$, $m_Q (P, l_{i, P}) \geq \beta_i > |L_i|$.

Since each $P \in J$ is of some type $i= i_P \in \{1, 2, \ldots, d\}$, there exists some popular $i_0$ such that at least $\frac{|J|}{d}$ points in $J$ are of type $i_0$. Assuming such points form a set $J_{i_0}$. Then by the discussion above and Lemma \ref{bezoutineq},
\begin{equation}
|J||L_{i_0}| \lesssim |J_{i_0}||L_{i_0}| \leq \sum_{P \in J_{i_0}} m_Q (P, l_{i_0, P})\leq \sum_{P \in J} m_Q (P, l_{i_0, P}) \leq |L_{i_0}|\deg Q \lesssim |J|^{\frac{1}{d}}\cdot (\prod_{i=1}^d |L_i|)^{\frac{1}{d}} \cdot |L_{i_0}|
\end{equation}
which is equivalent to $|J| \lesssim (\prod_{i=1}^d N_i)^{\frac{1}{d-1}}$.
\end{proof}

\begin{rem}
In the study of multilinear incidence geometry problems such as ones in this paper, it was noted for quite a while (\cite{carberycomm}) that it can be good to have the low degree Taylor series, in addition to the polynomial itself, vanishing at given points. This is the approach we take here and in the next section.
\end{rem}

\section{A proof of Carbery's conjecture}

In this section we prove Carbery's conjecture on counting multijoints with multiplicity (Theorem \ref{carberytheorem}), which generalizes the Multijoints Theorem \ref{multijoint} we just proved. Our proof will be based on the techniques we have developed so far.

Theorem \ref{carberytheorem} is implied by the following theorem on joints with multiplicity.

\begin{thm}\label{multiplicityjointsthm}
In $\mathbb{F}^d$, assuming we are given $N$ lines. Here we allow a same line to show up multiple times and denote the resulting set with multiplicity to be $L = (l_1, \ldots, l_N)$. For every point $P \in \mathbb{F}^d$, define the joint multiplicity $M(P) = |\{(l_{i_1}, \ldots, l_{i_d}): \{l_{i_j}\} \text{ form a joint at } P\}|$. Then
\begin{equation}
\sum_{P \in \mathbb{F}^d} M(P)^{\frac{1}{d-1}} \lesssim N^{\frac{d}{d-1}}.
\end{equation}
\end{thm}

\begin{proof}[Proof that Theorem \ref{multiplicityjointsthm} implies Theorem \ref{carberytheorem}]
Assuming we have Theorem \ref{multiplicityjointsthm} proved already and are given the assumption of Theorem \ref{carberytheorem}. Now consider the collection $L$ of all $d$ families $L_i$ with each line in $L_i$ repeated $\prod_{j \neq i} N_j$ times. Then $N = |L| \sim \prod_{i=1}^d N_i$. Moreover, at each $P$ each multijoint contributes $(\prod_{i=1}^d N_i)^{d-1}$ to the joint multiplicity $M(P)$ of $L$ at $P$. Apply Theorem \ref{multiplicityjointsthm} to $L$ and we have
\begin{equation}
\sum_{P \in \mathbb{F}^d} \mu(P)^{\frac{1}{d-1}}\cdot \prod_{i=1}^d N_i \leq \sum_{P \in \mathbb{F}^d} M(P)^{\frac{1}{d-1}} \lesssim (\prod_{i=1}^d N_i)^{\frac{d}{d-1}}
\end{equation}
and hence (\ref{carberytheoremeq}), as desired.
\end{proof}

In order to prove Theorem \ref{multiplicityjointsthm}, we do some preliminary work to understand $M(P)$ better.

\begin{defn}
Assuming $L = (l_1, \ldots, l_N)$ is a set (with multiplicity) of lines in $\mathbb{F}^d$. For any $P \in \mathbb{F}^d$ and arbitrary integer $1 \leq j \leq d$, define \emph{the $j$th-minimum of $L$ at $P$} $r_j (P, L)$ to be
\begin{equation}
r_j (P, L) = \min_{V \text{ is a subspace of } \mathbb{F}^d, \dim V = j-1} |\{i: P \in l_i, l_i \text{ is not parallel to } V\}|.
\end{equation}
\end{defn}

Hence $r_1 (P, L)$ is simply the number of $l_i$'s that pass through $P$. As another example, $r_d (P, L) > 0$ is equivalent to saying that the lines in $L$ form at least one joint at $P$. It is also trivial that $r_1 (P, L) \geq \cdots \geq r_d (P, L)$. The reason that we call them $j$th-minimum is simply that they resemble the successive minima in the geometry of numbers a bit.

\begin{lem}\label{whatisMP}
\begin{equation}\label{whatisMPeq}
M(P) \sim \prod_{j=1}^d r_j (P, L).
\end{equation}
\end{lem}

\begin{proof}
Let us count the number of tuples $(l_{i_1}, \ldots, l_{i_d})$ forming a joint at $P$. Once $\{l_{i_k}\}_{1\leq k < j}$ are fixed we always have at least $r_j (P, L)$ different ways to choose $l_{i_j} \ni P$ not parallel to the $(j-1)$-dimensional subspace determined by $\{l_{i_k}\}_{1\leq k < j}$. In this way eventually the $d$ lines we choose will form a joint at $P$. Hence $M(P) \geq \prod_{i=1}^d r_j (P, L)$.

On the other hand, for $1\leq j \leq d$ choose $W_j (P)$ to be a $(j-1)$-dimensional space such that the set $X_j (P) = \{i : p \in l_i, l_i \text{ not parallel to } W_j (P)\}$ has cardinality $r_j (P, L)$. Then for any $l_{i_1}, \ldots, l_{i_d}$ forming a joint at $P$, there has to be at least one number among $i_1, \ldots, i_d$ that belongs to $X_d (P)$; at least two numbers among $i_1, \ldots, i_d$ that belong to $X_{d-1} (P)$, $\ldots$, at least $d$ numbers among $i_1, \ldots, i_d$ that belong to $X_{1} (P)$. Hence among $i_1, \ldots, i_d$ there is at least one number that belongs to $X_d (P)$; among the rest there is at least one that belongs to $X_{d-1} (P)$, $\ldots$, in the end the one number left belongs to $X_{1} (P)$. Hence $M(P) \lesssim \prod_{j=1}^d r_j (P, L)$.
\end{proof}

The proof of Theorem \ref{multiplicityjointsthm} is then a result of a good understanding on the lower bound $m_Q (P, l)$ for an \emph{arbitrary} $l$ passing through $P$. We want $m_Q (P, l)$ to be large on most directions and are able to prove that (intuitively) when it is small the line $l$ is usually parallel to certain ``bad'' subspaces (note that we do not run into more complicated ``bad subvarieties situation'' in our proof at all, which could be somewhat counterintuitive). This is then good enough to prove Theorem \ref{multiplicityjointsthm}.

\begin{proof}[Proof of Theorem \ref{multiplicityjointsthm}]
For every point $P$ such that $M(P) > 0$, we claim that we can find a flag $V_1 (P) \subseteq V_2 (P) \subseteq \cdots \subseteq V_d (P) \subseteq V_{d+1} (P) = \mathbb{F}^d$, such that (a) $\dim V_j (P) = j-1, 1\leq j \leq d$; (b) the set $L_j (P) = \{i : p \in l_i, l_i \text{ not parallel to } V_j (P)\}$ has cardinality $\sim r_j (P, L)$ for $1 \leq j \leq d$ and (c) for any $1 \leq j \leq d$ and any $(j-1)$-dimensional subspace $V$ of $V_{j+1} (P)$, we have
\begin{equation}\label{extremeofVj}
|\{i : p \in l_i, l_i \text{ is parallel to } V\}| \leq |\{i : p \in l_i, l_i \text{ is parallel to } V_j (P)\}|.
\end{equation}

Note that this is slightly stronger than what we had for $W_j (P)$'s and $X_j (P)$'s in the proof of Lemma \ref{whatisMP}.

In this paragraph we prove the above claim. Our strategy is to choose $\{V_j (P)\}$ inductively (downwards): By the definition of $r_d (P, L)$, there exists a $(d-1)$-dimensional subspace $V_d (P)$ such that the above defined $L_d (P)$ has cardinality $r_d (P, L)$. Assuming we already have $V_{j+1} (P)\subseteq \cdots \subseteq V_d (P) (1\leq j < d)$ such that (a), (b) and (c) hold for them. Now we look for $V_j (P)$. By the definition of $r_j (P, L)$, there exists $V_j ' (P)$ being a $(j-1)$-dimensional subspace such that the set $\{i : p \in l_i, l_i \text{ not parallel to } V_j ' (P)\}$ has cardinality $r_j (P, L)$. We choose $V_j '' (P)$ to be an arbitrary $(j-1)$-dimensional subspace of $V_{j+1} (P)$ containing $V_j ' (P) \bigcap V_{j+1} (P)$. Now if a line is not parallel to $V_j '' (P)$, then it is either not parallel to $V_j ' (P)$ or not parallel to $V_{j+1} (P)$. The number of the lines of the first type is $\leq r_j (P, L)$ while by induction hypothesis the number of the lines of the second type is $\lesssim r_{j+1} (P, L) \leq r_j (P, L)$. Hence the set $\{i : p \in l_i, l_i \text{ not parallel to } V_j '' (P)\}$ has cardinality $\lesssim r_j (P, L)$. Now choose $V_j (P) \subseteq V_{j+1} (P)$ such that $\dim V_j (P) = j-1$ and $|\{i : P \in l_i, l_i \text{ parallel to } V_j (P)\}|$ is maximal possible. Then it is obvious that $|\{i : p \in l_i, l_i \text{ not parallel to } V_j (P)\}| \leq |\{i : p \in l_i, l_i \text{ not parallel to } V_j '' (P)\}| \lesssim r_j (P, L)$ and $V_j(P)$ satisfies (c) as well as (a) and (b). This closes the induction and our claim follows.

Now for every $P$ such that $M(P) > 0$, we choose $d$ lines $s_{1, P}, \ldots, s_{d, P}$ passing through $P$ such that $V_j (P)$ is spanned by the directions of $s_{k, P}, 1 \leq k < j$. Notice that $s_{j, P}$ may not belong to $L$ anymore. We then fix a linear transform $T_P$ that sends $P$ to the origin and $s_{j, P}$ to the $j$-th coordinate axis.

We look for a nonzero polynomial $Q$ such that $(T_P^{-1})^* Q$, when expanded as a sum of monomials (Taylor series), has no term of the form $x_1^{\beta_1} \cdots x_d^{\beta_d}$ with $\beta_j \leq \frac{100^j B\cdot\prod_{k \neq j} r_k(P, L)}{(\prod_{k=1}^d r_k (P, L))^{\frac{d-2}{d-1}}}$ simultaneously holding. Here $B$ is a large positive integer depending on $L$ such that all $\frac{B\cdot\prod_{k \neq j} r_k(P, L)}{(\prod_{k=1}^d r_k (P, L))^{\frac{d-2}{d-1}}} > 100d^{100}$, $\forall P, j$ with $M(P) > 0$. By Lemma \ref{whatisMP} and parameter counting, there exists such a nonzero $Q$ with $\deg Q \lesssim B\cdot (\sum_{P \in \mathbb{F}^d} M(P)^{\frac{1}{d-1}})^{\frac{1}{d}}$.

Next we prove the following crucial estimate for every $P \in \mathbb{F}^d$ satisfying $M(P) > 0$:
\begin{equation}\label{contributionofP}
\sum_{i: P \in l_i} m_Q (P, l_i) \gtrsim B\cdot M(P)^{\frac{1}{d-1}}.
\end{equation}

To prove (\ref{contributionofP}), we fix $P$. Write
\begin{equation}
(T_P^{-1})^* Q = \sum_{\beta} c_{\beta} x^{\beta}
\end{equation}
and then look at all $\beta$ such that $c_{\beta} \neq 0$ and $|\beta|$ smallest possible. Call all such $\beta$ to be \emph{lowest}. By the assumption of $Q$, there is some $1 \leq j_0 \leq d$ such that (i) for every lowest $\beta$ and every $j < j_0$, $\beta_j \leq \frac{100^j B\cdot\prod_{k \neq j} r_k(P, L)}{(\prod_{k=1}^d r_k (P, L))^{\frac{d-2}{d-1}}}$ and (ii) there exists a lowest $\beta$ such that $\beta_{j_0} > \frac{100^{j_0} B\cdot\prod_{k \neq j_0} r_k(P, L)}{(\prod_{k=1}^d r_k (P, L))^{\frac{d-2}{d-1}}}$. We fix $j_0$ in the rest of the proof of (\ref{contributionofP}).

Let $j_1\leq d$ be the largest integer $j$ such that
\begin{eqnarray}\label{spacecapacitycondition}
& |\{i : P \in l_i, l_i \text{ is parallel to } V_j (P) \text{ but not parallel to }V_{j_0} (P)\}|\nonumber\\
& \leq \frac{1}{2} |\{i : P \in l_i, l_i \text{ is parallel to } V_{j+1} (P) \text{ but not parallel to }V_{j_0} (P)\}|.
\end{eqnarray}

We notice $j = j_0$ satisfies (\ref{spacecapacitycondition}) as LHS would be $0$. Hence $j_1 \geq j_0$.

Note that $|\{i : P \in l_i, l_i \text{ is parallel to } V_{d+1} (P) \text{ but not parallel to }V_{j_0} (P)\}| = |L_{j_0} (P)| \sim r_{j_0} (P, L)$. By the opposite of (\ref{spacecapacitycondition}) from $j= j_1 + 1$ to $j = d$ we deduce
\begin{equation}
|\{i : P \in l_i, l_i \text{ is parallel to } V_{j_1 + 1} (P) \text{ but not parallel to }V_{j_0} (P)\}| \gtrsim r_{j_0} (P, L)
\end{equation}
which also holds when $j_1 = d$.

Since $V_{j_1 + 1} \subseteq V_{d+1}$ we also have the other side of the inequality:
\begin{equation}\label{middledimensionestimate}
|\{i : P \in l_i, l_i \text{ is parallel to } V_{j_1 + 1} (P) \text{ but not parallel to }V_{j_0} (P)\}| \sim r_{j_0} (P, L).
\end{equation}

Moreover by (\ref{spacecapacitycondition}) for $j = j_1$ we also have
\begin{equation}\label{step0}
|\{i : P \in l_i, l_i \text{ is parallel to } V_{j_1 + 1} (P) \text{ but not parallel to } V_{j_1}\}| \gtrsim r_{j_0} (P, L).
\end{equation}

By the assumption (c) in the beginning of our current proof of Theorem \ref{multiplicityjointsthm} and (\ref{step0}), for any $V_{j_0} (P) \subseteq V \subseteq V_{j_1+1} (P)$ with $\dim V =j_1 -1$,
\begin{equation}\label{step1}
|\{i : P \in l_i, l_i \text{ is parallel to } V_{j_1 + 1} (P) \text{ but not parallel to }V\}| \gtrsim r_{j_0} (P, L).
\end{equation}

By (\ref{middledimensionestimate}) we also have the other side of (\ref{step1}):
\begin{equation}\label{middledimensionlowerbd}
|\{i : P \in l_i, l_i \text{ is parallel to } V_{j_1 + 1} (P) \text{ but not parallel to }V\}| \sim r_{j_0} (P, L).
\end{equation}

Let $t$ = $j_1 - j_0 +1 > 0$. For any $t$ lines $l_{i_1}, \ldots, l_{i_t}$ such that the direction of them and $V_{j_0}$ exactly span $V_{j_1 + 1}$, we prove there exists $1 \leq h \leq t$ such that for $l = l_{i_h}$,
\begin{equation}\label{localestimateonjoints}
m_{Q} (P, l) \gtrsim \frac{B\cdot\prod_{k \neq j_0} r_k(P, L)}{(\prod_{k=1}^d r_k (P, L))^{\frac{d-2}{d-1}}}.
\end{equation}

In order to show (\ref{localestimateonjoints}), we choose a linear transform $T_P '$ that preserves the origin and the directions of the $j$-th coordinate axis for $1\leq j < j_0$ or $j_1 < j \leq d$, while sending ${(T_P)}_* l_{i_{j-j_0 + 1}}$ to the $j$-th coordinate axis for $j_0 \leq j \leq j_1$ (by assumption in the beginning of this paragraph, all $d$ directions we mentioned are linearly independent). Then $T_P ' \circ T_P$ sends $l_{i_{j-j_0 + 1}}$ to the $j$-th coordinate axis, $j_0 \leq j \leq j_1$. Moreover, $((T_P ' \circ T_P)^{-1})^* Q = ((T_{P} ')^{-1})^* (T_P^{-1})^* Q$. We assume that $\delta$ is a positive number such that among all the lowest order terms $\prod_{k=1}^d x_k^{\gamma_k}$ of $((T_P ' \circ T_P)^{-1})^* Q$, $\delta$ is the maximal possible sum $\gamma_{j_0} + \ldots + \gamma_{j_1}$. Now look at the action of $(T_P ')^*$ on this polynomial. By the assumptions, $(T_P ')^* x_j = x_j$ for $j_1 < j \leq d$, $(T_P ')^* x_j = x_j + H_j (x_{j_0}, \ldots, x_{j_1})$ for $1 \leq j < j_0$ where $H_j$ is some linear form, and finally $(T_P ')^* x_j = F_j (x_{j_0}, \ldots, x_{j_1})$ for $j_0 \leq j \leq j_1$ where the linear transformation $(x_{j_0}, \ldots, x_{j_1}) \mapsto (F_{j_0} (x_{j_0}, \ldots, x_{j_1}), \ldots, F_{j_1} (x_{j_0}, \ldots, x_{j_1}))$ is invertible. The action of its inverse $((T_P ')^{-1})^*$ on coordinate functions has the same form. Now by assumption (i), for every monomial $x^{\beta}$ with nonzero coefficient in the lowest homogeneous term of $((T_P)^{-1})^* Q$ and every $j < j_0$, $\beta_j \leq \frac{100^j B\cdot\prod_{k \neq j} r_k(P, L)}{(\prod_{k=1}^d r_k (P, L))^{\frac{d-2}{d-1}}}$. Hence this is also true for every monomial $x^{\beta}$ with nonzero coefficient in the lowest homogeneous term of $((T_P ' \circ T_P)^{-1})^* Q$. Apply $(T_P ')^*$ to $((T_P ' \circ T_P)^{-1})^* Q$ we actually get $(T_P^{-1})^* Q$. By the bounds on the exponents of $x_j (j \leq j_1)$ in the lowest homogeneous part of $((T_P ' \circ T_P)^{-1})^* Q$ we have above and the fact that $r_1 (P, L) \geq \cdots \geq r_d (P, L)$, we deduce that for any lowest order term $\prod_{k=1}^d x_k^{\beta_k}$ of $(T_P^{-1})^* Q$,
\begin{eqnarray}
\sum_{j=1}^{j_1} \beta_j &\leq \sum_{j=1}^{j_0 -1} \frac{100^j B\cdot\prod_{k \neq j} r_k(P, L)}{(\prod_{k=1}^d r_k (P, L))^{\frac{d-2}{d-1}}} + \delta\nonumber\\
& \leq \frac{2 \cdot 100^{j_0 -1} B\cdot\prod_{k \neq j_0} r_k(P, L)}{(\prod_{k=1}^d r_k (P, L))^{\frac{d-2}{d-1}}} + \delta.
\end{eqnarray}

But by the assumption (ii) there exists such a $\beta$ satisfying $\beta_{j_0} > \frac{100^{j_0} B\cdot\prod_{k \neq j_0} r_k(P, L)}{(\prod_{k=1}^d r_k (P, L))^{\frac{d-2}{d-1}}}$. We deduce that $\delta \gtrsim \frac{B\cdot\prod_{k \neq j_0} r_k(P, L)}{(\prod_{k=1}^d r_k (P, L))^{\frac{d-2}{d-1}}}$. This means among all the lowest order terms $\prod_{k=1}^d x_k^{\gamma_k}$ of $((T_P ' \circ T_P)^{-1})^* Q$, there is at least one with $\gamma_{j_0} + \ldots + \gamma_{j_1} \gtrsim \frac{B\cdot\prod_{k \neq j_0} r_k(P, L)}{(\prod_{k=1}^d r_k (P, L))^{\frac{d-2}{d-1}}}$. Note that $T_P ' \circ T_P$ sends $l_{i_{j-j_0 + 1}}$ to the $j$-th coordinate axis, $j_0 \leq j \leq j_1$, by Lemma \ref{lowerbdofm} we deduce that (\ref{localestimateonjoints}) has to hold for some $l = l_{i_h}$, $1 \leq h \leq t$.

Define $Y(P) = \{i : P \in l_i, l_i \text{ is parallel to } V_{j_1 + 1} (P) \text{ but not parallel to }V_{j_0} (P)\}$ Look at all the lines $l = l_i$ violating (\ref{localestimateonjoints}) with $i \in Y(P)$. By the above discussion, there are never $(j_1 - j_0 + 1)$ such lines satisfying their directions and $V_{j_0}$ span $V_{j_1}$. Therefore all the directions of such lines have to be contained in a common $(j_1 -1)$-dimensional subspace $\widetilde{V} (P)$. Now by (\ref{middledimensionlowerbd}) with $V = \widetilde{V} (P) \supseteq V_{j_0}$ we deduce that there are $\gtrsim r_{j_0} (P, L)$ different $i\in Y(P)$ such that $l = l_i$ satisfies (\ref{localestimateonjoints}). Summing (\ref{localestimateonjoints}) over all of them and note that $M(P) \sim \prod_{k=1}^d r_k (P, L)$, we obtain (\ref{contributionofP}).

The rest of the proof is simple. Summing (\ref{contributionofP}) over all $P \in \mathbb{F}^d$ (it certainly holds for those $P$ with $M(P) = 0$ too) and use Lemma \ref{bezoutineq}, we deduce
\begin{eqnarray}
B\cdot \sum_{P \in \mathbb{F}^d} M(P)^{\frac{1}{d-1}} \lesssim &\sum_{P\in \mathbb{F}^d} \sum_{i: P \in l_i} m_Q (P, l_i)\nonumber\\
= & \sum_{1\leq i\leq N} \sum_{P \in l_i} m_Q (P, l_i)\nonumber\\
\leq & N \cdot \deg Q\nonumber\\
\lesssim & B\cdot N \cdot (\sum_{P \in \mathbb{F}^d} M(P)^{\frac{1}{d-1}})^{\frac{1}{d}}
\end{eqnarray}
which is equivalent to the conclusion of the theorem.
\end{proof}

\bibliographystyle{amsalpha}
\bibliography{stageref}

Department of Mathematics, Princeton University, Princeton, NJ 08540

ruixiang@math.princeton.edu

\end{document}